\documentclass[10pt,a4paper,twoside]{amsart}
\usepackage{amsmath,amsfonts,amsthm,amsopn,color,amssymb,enumitem}
\usepackage{palatino}
\usepackage{graphicx}
\usepackage[colorlinks=true]{hyperref}
\hypersetup{urlcolor=blue, citecolor=red, linkcolor=blue}

\newcommand{\e}{\varepsilon}

\newcommand{\C}{\mathbb{C}}
\newcommand{\R}{\mathbb{R}}

\newcommand{\RD}{{\mathbb{R}^2}}

\renewcommand{\le}{\leslant}
\renewcommand{\ge}{\geslant}
\renewcommand{\a }{\alpha }

\renewcommand{\b }{\beta }

\newcommand{\I}{\mathcal{I}}

\renewcommand{\C}{\mathbb{C}}
\renewcommand{\o}{\omega}

\def\bbm[#1]{\mbox{\boldmath $#1$}}
\newcommand{\beq }{\begin{equation}}
\newcommand{\eeq }{\end{equation}}

\newcommand{\dis}{\displaystyle}

\renewcommand{\le}{\leqslant}
\renewcommand{\ge}{\geqslant}

\newcommand{\ir}{\int_{-\infty}^{+\infty}}
\newcommand{\ird }

\newtheorem{theorem}{Theorem}[section]
\newtheorem{lemma}[theorem]{Lemma}

\newtheorem{proposition}[theorem]{Proposition}
\newtheorem{remark}[theorem]{Remark}
\newtheorem{corollary}[theorem]{Corollary}

\title[Boundary concentration of a Gauged Nonlinear Schr\"{o}dinger Equation]{Boundary concentration of a Gauged Nonlinear Schr\"{o}dinger Equation on large balls}

\author[Pomponio]{Alessio Pomponio$^1$}
\address{$^1$Dipartimento di Meccanica, Matematica e Management, Politecnico di
Bari, Via E. Orabona 4, 70125 Bari, Italy.}
\author[Ruiz]{David Ruiz$^2$}
\address{$^2$Dpto. An\'{a}lisis Matem\'{a}tico, Granada, 18071 Spain.}

\thanks{A.P. is supported by M.I.U.R. - P.R.I.N.
``Metodi variazionali e topologici nello studio di fenomeni non
lineari'', by GNAMPA Project ``Metodi Variazionali e Problemi
Ellittici Non Lineari'' and by FRA2011 ``Equazioni ellittiche di
tipo Born-Infeld". D.R. is supported by the Spanish Ministry of
Science and Innovation under Grant MTM2011-26717 and by J.
Andalucia (FQM 116).}

\email{a.pomponio@poliba.it, daruiz@ugr.es}
\date{}

\keywords{Gauged Schr\"{o}dinger Equations,
Lyapunov-Schmidt reduction.}

\subjclass[2010]{35J20, 35Q55.}

\begin{document}

\maketitle


\begin{abstract}

This paper is motivated by a gauged Schr\"{o}dinger equation in
dimension 2 including the so-called Chern-Simons term. The
radially symmetric case leads to an elliptic problem with a
nonlocal defocusing term, in competition with a local focusing
nonlinearity. In this work we pose the equations in a ball under
homogeneous Dirichlet boundary conditions. By using singular
perturbation arguments we prove existence of solutions for large
values of the radius. Those solutions are located close to the boundary
and the limit profile is given.

\end{abstract}

\section{Introduction}

Let us consider a gauged Nonlinear Schr\"{o}dinger Equation in the
plane: \beq \label{planar} i D_0\phi + (D_1D_1 +D_2D_2)\phi +
|\phi|^{p-1} \phi =0. \eeq Here $t \in \R$, $x=(x_1, x_2) \in
\R^2$, $\phi : \R \times \R^2 \to \C$ is the scalar field, $A_\mu
: \R\times \R^2 \to \R$ are the components of the gauge potential
and $D_\mu =
\partial_\mu + i A_\mu$ is the covariant derivative ($\mu = 0,\
1,\ 2$).

The natural equations for the gauge potential $A_\mu$ are the
Maxwell equations. However, the modified gauge field equation
proposes to include also the so-called Chern-Simons term, that is,

\[ 
\partial_\mu F^{\mu \nu} + \frac 1 2
\kappa \epsilon^{\nu \alpha \beta} F_{\alpha \beta}= j^\nu, \
\mbox{ with } \ F_{\mu \nu}=
\partial_\mu A_\nu - \partial_\nu A_\mu. \]

In the above equation, $\kappa$ is a parameter that measures the
strength of the Chern-Simons term. As usual, $\epsilon^{\nu \alpha
\beta}$ is the Levi-Civita tensor, and super-indices are related
to the Minkowski metric with signature $(1,-1,-1)$. Finally,
$j^\mu$ is the conserved matter current,

$$ j^0= |\phi|^2, \ j^i= 2 {\rm Im} \left (\bar{\phi} D_i \phi \right).$$

At low energies, the Maxwell term becomes negligible and can be
dropped, giving rise to: \beq \label{cs} \frac 1 2 \kappa
\epsilon^{\nu \alpha \beta} F_{\alpha \beta}= j^\nu. \eeq See
\cite{hagen, hagen2, jackiw0, jackiw, jackiw2} and \cite[Chapter
1]{tar} for the discussion above.

For the sake of simplicity, let us normalize so that $\kappa = 2$.
Equations \eqref{planar} and \eqref{cs} lead us to the problem:
 \beq \label{eq:e0}
\begin{array}{l}
i D_0\phi + (D_1D_1 +D_2D_2)\phi + |\phi|^{p-1} \phi =0,\\
\partial_0 A_1  -  \partial_1 A_0  = {\rm Im}( \bar{\phi}D_2\phi), \\
\partial_0 A_2  - \partial_2 A_0 = -{\rm Im}( \bar{\phi}D_1\phi), \\
\partial_1 A_2  -  \partial_2 A_1 =  \frac 1 2 |\phi|^2. \end{array}
\eeq

As usual in Chern-Simons theory, problem \eqref{eq:e0} is
invariant under gauge transformation,
\[
\phi \to
\phi e^{i\chi}, \quad A_\mu \to A_\mu - \partial_{\mu} \chi
\]
for any arbitrary $C^\infty$ function $\chi$.

This model was first proposed and studied in \cite{jackiw0,
jackiw, jackiw2}, and sometimes has received the name of
Chern-Simons-Schr\"{o}dinger equation. The initial value problem,
as well as global existence and blow-up, has been addressed in
\cite{berge, huh, huh3} for the case $p=3$.

The existence of stationary states for \eqref{eq:e0} and general
$p>1$  has been studied recently in \cite{byeon, huh2, noi}. By
using the ansatz:
\begin{equation*}
\begin{array} {lll} \phi(t,x) = u(|x|) e^{i \omega t}, && A_0(x)= A_0(|x|),
\\
A_1(t,x)=\dis -\frac{x_2}{|x|^2}h(|x|), && A_2(t,x)= \dis
\frac{x_1}{|x|^2}h(|x|),
\end{array}
\end{equation*}
and after some manipulations, in \cite{byeon} it is found that:
\beq \label{equation}   - \Delta u(x) + \left( \o + \dis
\frac{h^2(|x|)}{|x|^2} +  \int_{|x|}^{+\infty} \frac{h(s)}{s}
u^2(s)\, ds   \right)  u(x)  =|u(x)|^{p-1}u(x), \quad \hbox{in
}\RD, \eeq where
$$h(r)= \frac 12\int_0^r s u^2(s)  \, ds.$$
Here $A_0$ takes the form:

$$A_0(r)=  \int_r^{+\infty} \frac{h(s)}{s} u^2(s)\, ds.$$

Observe that \eqref{equation} is a nonlocal equation. Moreover, in
\cite{byeon} it is shown that \eqref{equation} is indeed the
Euler-Lagrange equation of the energy functional:
$$ I : H_r^1(\R^2) \to \R, $$ defined as
\begin{align*}
I (u) & = \dis \frac 12 \int_{\R^2} \left(|\nabla u(x)|^2 +
\o u^2(x) \right) \, dx \nonumber
\\
 & \quad+\dis \frac{1}{8} \int_{\R^2}
\frac{u^2(x)}{|x|^2}\left(\int_0^r s u^2(s) \, ds\right)^2 dx  -
\dis \frac{1}{p+1}  \int_{\R^2} |u(x)|^{p+1} \, dx.
\end{align*}
Here $H^1_r(\R^2)$ is the Sobolev space of radially symmetric
functions. In \cite{byeon, huh2, noi} critical points of
$I$ are found by using variational methods. It is shown
that the value $p=3$ is critical for this problem. Indeed, for
$p>3$ the energy functional is unbounded from below and satisfies
a mountain-pass geometry. Even if the (PS) property is not known
to be satisfied, positive solutions are found in \cite{byeon}.

Instead, for $p \in (1,3)$ the global behavior of $I$ is
different and depends on the parameter $\o$. In \cite{noi} we show
that $I$ is bounded from below if and only if $\o \ge
\o_0$, where the threshold value $\o_0$ is explicit, namely:

\[ 
\omega_0= \frac{3-p}{3+p}\
3^{\frac{p-1}{2(3-p)}}\ 2^{\frac{2}{3-p}} \left(
\frac{m^2(3+p)}{p-1}\right)^{-\frac{p-1}{2(3-p)}}. \] Here $m = \ir w_1^2(r)\, dr$, and $w_1$ is the unique positive even solution of the problem $-w'' + w= w^p$ in $\R$, i.e.,
\beq \label{blabla}  w_1(r)= \left ( \frac{2}{p+1} \cosh^2 \left
(\frac{p-1}{2} r \right) \right)^{\frac{1}{1-p}}.\eeq

In this paper we study positive solutions for the Dirichlet
problem posed on a ball $B(0,R)$, that is:

\beq \label{equation0} \left \{ \begin{array}{ll}  - \Delta u(x) +
\left( \o + \dis \frac{h^2(|x|)}{|x|^2} +  \int_{|x|}^{R}
\frac{h(s)}{s} u^2(s)\, ds   \right)  u(x)  =u^p(x),&  x \in
B(0,R), \\
u(x)=0, & x \in \partial B(0,R),
\end{array} \right. \eeq where, as above,
$$h(r)= \frac 12\int_0^r s u^2(s)  \, ds.$$ Let us denote by
$H_R= H^1_{0,r}(B(0,R))$ the Sobolev Space of radially symmetric
functions. Again, \eqref{equation0} is the Euler-Lagrange equation
of the functional $I_R : H_R \to \R$, $I_R = {I}|_{H_R}$.

A first motivation for this study comes from the proofs of
\cite{noi}. It is not difficult to show that $I_R$ is coercive for any $R$ fixed. So, there exists a
minimizer $u_R$ of $I_R $. In order to prove boundedness of $u_R$
when $R\to +\infty$, the problem is the possible loss of mass at
infinity. In \cite{noi} we prove that, if unbounded, the sequence
$u_R$ actually behaves as a soliton, when $u_R$ is interpreted as
a function of a single real variable. Then, the energy functional
$I$ admits a natural approximation through a convenient
limit functional $J$, see Section 2 for details. Moreover, the
solutions of that limit functional, and their energy, can be found
explicitly, and this leads us to a contradiction if  $\o \ge
\o_0$.

In this paper we show that this kind of solutions actually exists
for $\o < \o_0$. Moreover, their energy tends to $-\infty$ as $R
\to +\infty$, and they seem to be local minima of $I$. Therefore,
they could correspond to global minima of $I$, which is an
interesting phenomenon: the minima of the energy tend to aggregate
near the boundary of the domain.

A second motivation comes from reference \cite{amni}. Observe that
for any $u$, the function:

$$ \dis \frac{h^2(|x|)}{|x|^2} +  \int_{|x|}^{R}
\frac{h(s)}{s} u^2(s)\, ds,   $$ is decreasing in $|x|$.
Therefore, a solution like in \cite{amni} could exist. However
here the problem is completely different, since $h$ itself depends
on $u$ in a nonlocal fashion. Indeed, natural restrictions on $\o$
and $p$ appear, as commented above.

The main result of this paper is the following:

\begin{theorem} \label{main} Assume $p \in (1,3)$, $\o \in
(0,\o_0)$. Then, for $R$ large enough, problem \eqref{equation0}
admits a positive solution $\tilde{u}_R \sim U(r-\rho)$, with
$$ U(r)=  k_2^{\frac{1}{p-1}}w_1(\sqrt{k_2}r). $$
Here $w_1$ is given by \eqref{blabla} and $k_2$ is the largest
positive root of the equation:
$$ k= \o + \frac 1 4 m^2 k^{\frac{5-p}{p-1}},\ k>0.
$$ Finally,
$$\rho \sim R- \frac{1}{2 \sqrt{k_2}} \log R.$$

\end{theorem}

As commented above, the assumption $\o \in (0,\o_0)$ is natural
for this result. Moreover, observe that $\o_0$ is well defined
only for $p \in (1,3)$, and it tends to $0$ as $p$ tends to $3$.

The proof uses a singular perturbation method based on a
Lyapunov-Schmidt reduction, which has become a well-known
technique in this kind of problems. However, there are some new
aspects on the problem under consideration. First, our problem
involves a nonlocal term which makes the argument of the proof
more complicated. We point out that it is not just a technical
difficulty, since the existence result depends heavily on the
nonlocal term. Indeed, those solutions do not exist for the autonomous Nonlinear Schr\"{o}dinger Equation, and are a special
feature of the model studied in this paper.

Moreover, the limit functional is different from the usual one in
Nonlinear Schr\"{o}dinger Equations and contains also a nonlocal term. As a consequence, we also need
to prove the non-degeneracy of the limit solution, which in this
case corresponds to a global minimum of the limit functional.

The rest of the paper is organized as follows. Section 2 is
devoted to some preliminary results. The limit functional and its
properties are exposed here. Moreover, we prove that the limit
solution is non-degenerate. The last two sections are devoted to the proof of
Theorem \ref{main}. Section 3 is devoted to solve the auxiliary equation in the Lyapunov-Schmidt reduction, whereas in Section 4 we prove the existence of a minimum for the reduced functional.

\subsection*{Acknowledgement}
This work has been partially carried out during a stay of A.P. in
Granada. He would like to express his deep gratitude to the
Departamento de An\'{a}lisis Matem\'{a}tico for the support and warm
hospitality.

\section{Preliminaries}

Let us first fix some notations. We denote by $H_R=H_{0,r}^1(B(0,R))$ the Sobolev space of radially symmetric functions. We use $\langle \cdot \rangle$, $\|
\cdot \|$ to denote its usual scalar product and norm, whereas
other norms, like Lebesgue norms in $B(0,R)\subset \R^2$, will be indicated with a
subscript. Moreover, $\langle \cdot \rangle_{H^1(a,b)}$,
$\|\cdot\|_{H^1(a,b)}$ are used to indicate the scalar product and
norm of the Sobolev space of dimension $1$. If nothing is
specified, strong and weak convergence of sequences of functions
are assumed in the space $H_R$.

In our estimates, we will frequently denote by $C>0$, $c>0$ fixed
constants, that may change from line to line, but are always
independent of the variable under consideration. We also use the
notations $O(1), o(1), O(\e), o(\e)$ to describe the asymptotic
behaviors of quantities in a standard way. Finally the letter
$x$ indicates a two-dimensional variable, and $r$, $s$, $t$ denote
one-dimensional variables.

\

Let us start with the following proposition:
\begin{proposition} $I_R $ is a $C^1$ functional, and its critical
points correspond to classical solutions of \eqref{equation0}.
\end{proposition}

This was proved in \cite{byeon} for the problem set in $\R^2$;
the same arguments apply to our case. Moreover, the derivative of $I_R$ is:
\begin{align} \label{versione 1}
\frac{I_R'(z)[u]}{2\pi} &= \int_0^R (z'u' + \o z u - |z|^{p-2}z u)
r\, dr \nonumber
\\
&\quad+ \frac{1}{4}\int_0^R \frac{z(r) u(r)}{r}
\left(\int_0^r z^2(s) s\, ds \right)^2dr \nonumber \\
&\quad+ \frac{1}{2}\int_0^R \frac{z^2(r)}{r} \left(\int_0^r
z^2(t)t\, dt \right)\left(\int_0^r z(s)u(s)s\, ds \right) dr.
\end{align}

If in the last term we use Fubini Theorem in the variables $r$,
$s$, and after that we interchange the name of the variables, we
obtain:
\begin{align} \label{versione 2}
\frac{I_R'(z)[u]}{2\pi} &= \int_0^R (z'u' + \o z u - |z|^{p-2}z u)
r\, dr \nonumber
\\
&\quad+ \frac{1}{4}\int_0^R \frac{z(r) u(r)}{r}
\left(\int_0^r z^2(s) s\, ds \right)^2 dr \nonumber \\
&\quad+ \frac{1}{2}\int_0^R z(r)u(r) r \left(\int_r^R
\frac{z^2(s)}{s}\left(\int_0^s z^2(t)t\, dt
\right) ds \right) dr.
\end{align}
In this paper we will make use of both expressions for $I_R'$,
depending on the case.

\medskip Our proof of Theorem \ref{main} uses a Lyapunov-Schmidt
reduction to prove existence of a critical point for $I_R$. In order
to do that, some knowledge of the limit functional is needed.
First of all, let us give a heuristic derivation of this energy
functional. Consider $u(r)$ a fixed function, and define
$u_\rho(r)= u(r-\rho)$. Let us now estimate $I(u_\rho)$
as $\rho \to +\infty$; after the change of variables $r \to
r+\rho$, we obtain:
\begin{align*}
\frac{I (u_\rho)}{2\pi} &= \dis \frac 12 \int_{-\rho}^{+\infty}
(|u'|^2 + \o u^2)(r+\rho) \, dr
\\
& \quad+\dis \frac{1}{8} \int_{-\rho}^{\infty}
\frac{u^2(r)}{r+\rho}\left(\int_{-\rho}^r (s+\rho) u^2(s) \,
ds\right)^2 dr  - \dis \frac{1}{p+1}  \int_{-\rho}^{\infty}
|u|^{p+1}(r+\rho) \, dr.
\end{align*}

We estimate the above expression by simply replacing the
expressions $(r+\rho)$, $(s+\rho)$ with the constant $\rho$:
$$(2\pi)^{-1} I (u)  $$$$\sim  \dis \rho \left [ \frac 12 \ir (|u|'^2
+ \o u^2) \, dr
 +\dis \frac{1}{8} \ir
u^2(r) \left(\int_{-\infty}^r  u^2(s) \, ds\right)^2 dr  - \dis
\frac{1}{p+1}  \ir |u|^{p+1} \, dr \right ]$$
$$ =\rho \left [ \frac 12 \ir (|u|'^2
+ \o u^2) \, dr
 +\dis \frac{1}{24} \left ( \ir
u^2 dr \right)^3  - \dis \frac{1}{p+1}  \ir |u|^{p+1} \, dr \right
]. $$

Therefore, it is natural to consider the limit functional $J:
H^1(\R) \to \R$,

\beq \label{def J} J (u)= \frac 1 2  \ir \left ( |u'|^2 +
\omega u^2 \right ) dr + \frac{1}{24} \left ( \ir u^2 \, dr
\right)^3 - \frac{1}{p+1} \ir |u|^{p+1}\, dr. \eeq  Clearly, the
Euler-Lagrange equation of \eqref{def J} is the following limit
problem: \beq \label{limit}
 -u'' + \omega u + \frac 1 4 \left ( \ir
u^2(s)\, ds \right )^2 u = |u|^{p-1} u, \quad \hbox{in }\R. \eeq

Before going on with the study of \eqref{limit}, we need some
well-known facts concerning the problem \beq \label{wk} -w'' + k w
= w^p \quad \hbox{ in }\R,\ k>0.\eeq We denote by $w_k \in
H^1(\R)$ ($k>0$) the unique positive even solution of \eqref{wk}.
Observe that equation \eqref{wk} is integrable and the Hamiltonian
of $w_k$ is equal to $0$, that is,

\beq \label{hamiltonian}
 - \frac 1 2 |w_k'(r)|^2 + \frac k 2
w_k^2(r) - \frac{1}{p+1} w_k^{p+1}(r)=0, \ \hbox{ for all } r \in
\R. \eeq

It is also known that any positive solution of \eqref{wk} is of
the form $u(r)= w_k(r-\xi)$, for some $\xi\in \R$. Moreover, \beq
\label{scaling} w_k(r)= k^{\frac{1}{p-1}}w_1(\sqrt{k}r), \quad
\hbox{and} \quad w_1(r)= \left ( \frac{2}{p+1} \cosh^2 \left
(\frac{p-1}{2} r \right) \right)^{\frac{1}{1-p}}. \eeq
Notice that $w_k$ and $w_k'$ decay exponentially to zero at infinity.

In what follows we define \beq  \label{defm} m= \ir w_1^2\,dr.
\eeq

The following relations are also well known, and can be deduced
from \eqref{hamiltonian}: \beq \label{relations} \ir |w_1'|^2\, dr
= \frac{p-1}{p+3} m, \qquad \ir w_1^{p+1}\, dr =
\frac{2(p+1)}{p+3} m. \eeq

Let us now come back to equation \eqref{limit}. Consider $u$ a
positive solution of \eqref{limit}, and define $k= \omega + \frac
1 4 \left( \ir u^2\, dr\right )^2$. Then, $u$ is a solution of
$-u'' + k u = u^p$, so $u(r)= w_k(r-\xi)$ for some $\xi \in \R$.
By using \eqref{scaling}, we obtain:
$$ k= \omega + \frac 1 4 \left( \ir w_k^2(r)\, dr\right )^2=
\omega + \frac 1 4 k^{\frac{4}{p-1}} \left( \ir w_1^2(\sqrt{k}r)\,
dr\right )^2.$$ A change of variables leads us to the identity:

\beq \label{eq-k} k= \o + \frac 1 4 m^2 k^{\frac{5-p}{p-1}},\ k>0.
\eeq

Therefore, the existence of solutions for \eqref{limit} reduces to
the existence of solutions of the algebraic equation \eqref{eq-k}.
Moreover, we are also interested in the energy of those solutions,
as we shall see in Section 4. Those questions have been treated in
\cite[Section 3]{noi}, where the following results were obtained:

\begin{proposition} \label{explicit} Assume $p \in (1,3)$ and take $m$ as in \eqref{defm}.
Define:
\begin{align*}
\omega_0&= \frac{3-p}{3+p}\ 3^{\frac{p-1}{2(3-p)}}\
2^{\frac{2}{3-p}} \left(
\frac{m^2(3+p)}{p-1}\right)^{-\frac{p-1}{2(3-p)}}, 
\\
\omega_1&= \left( \frac{(5-p)m^2}{4(p-1)}
\right )^{-\frac{p-1}{2(3-p)}} - \frac{m^2}{4} \left(
\frac{(5-p)m^2}{4(p-1)}  \right )^{-\frac{(5-p)}{2(3-p)}}.
\end{align*}
%
%
%

The following holds:
\begin{enumerate}

\item $0<\o_0 < \o_1$;

\item if $\o > \o_1$, equation \eqref{eq-k} has no solution and
there is no nontrivial solution of \eqref{limit};

\item if $\o =\o_1$, equation \eqref{eq-k} has only one solution
$k_0$ and $w_{k_0}(r)$ is the only positive solution of
\eqref{limit} (apart from translations); \item if $\o \in (0,
\o_1)$, equation \eqref{eq-k} has two solutions $k_1(\o)<k_2(\o)$
and $w_{k_1}(r),w_{k_2}(r)$ are the only two positive solutions of
\eqref{limit} (apart from translations);

\item for any $\o \in (0,\o_1)$, $J(w_{k_1}) > 0$. \item
$J(w_{k_2})<0$ if and only if $\o \in (0,\o_0)$. In such case,
$w_{k_2}$ is a global minimizer for $J$.

\end{enumerate}

\end{proposition}

\begin{remark} In general, we cannot obtain a more explicit expression of $m$ depending on
$p$, but it can be easily approximated by using some software. In
Figure 1 the maps $\o_0(p)$ and $\o_1(p)$ have been plotted.

\begin{figure}[h]
\centering \fbox{
\begin{minipage}[c]{110mm}
           \centering
        \resizebox{99mm}{66mm}{\includegraphics{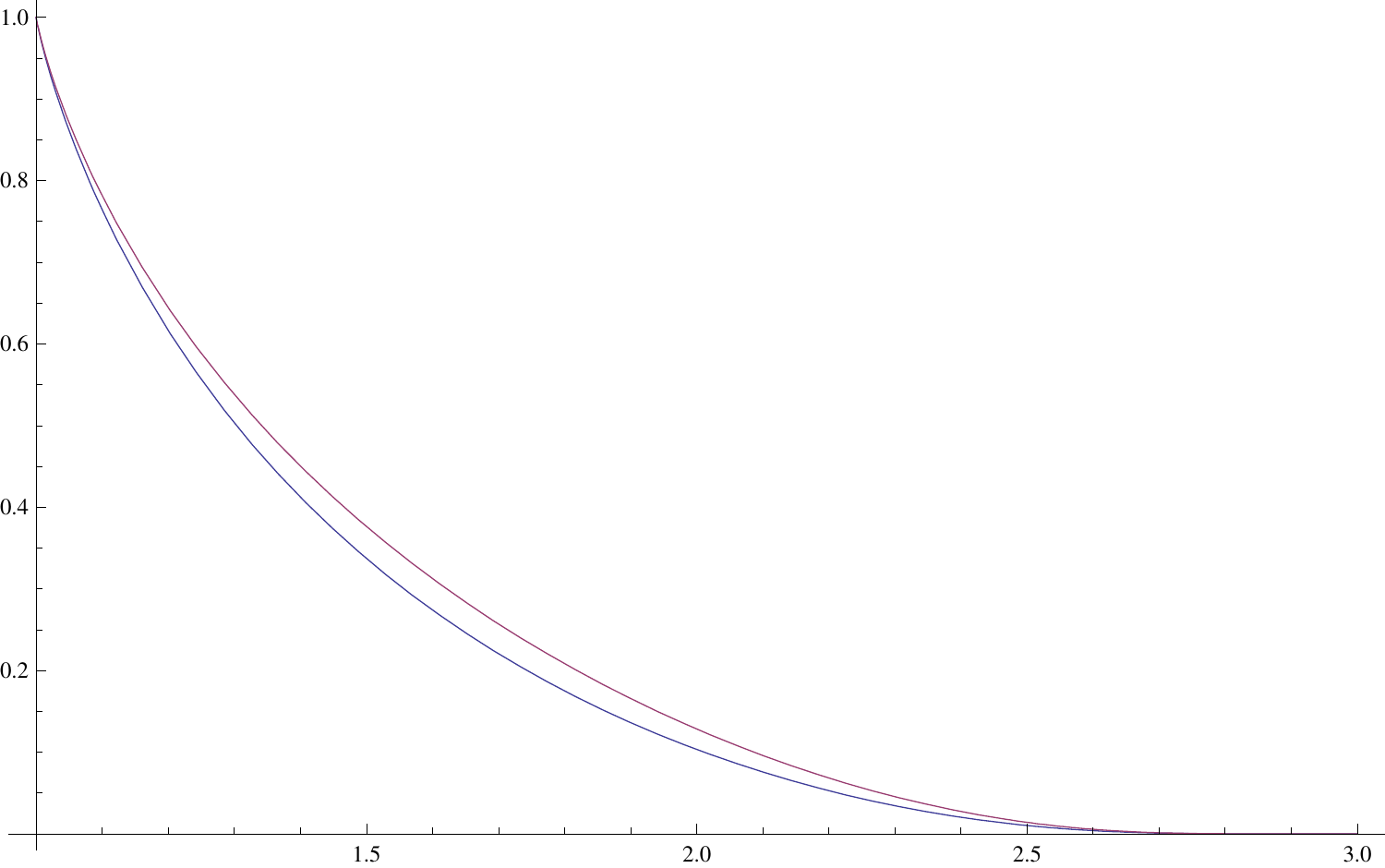}}
       \caption{\footnotesize The values $\o_0(p) <\o_1(p)$, for $p\in (1,3)$.}

         \end{minipage}
       }
\end{figure}

For some specific values of $p$, $m$ can be explicitly computed,
and hence $\o_0$ and $\o_1$. For instance, if $p=2$, $m=6$,
$\o_1=\frac{2}{9\sqrt{3}}$ and $\o_0=\frac{2}{5\sqrt{15}}$.

\end{remark}

As commented in the introduction, in Section 3 we will use the
above solutions as a limit profile of solutions of
\eqref{equation0} via a Lyapunov-Schmidt reduction. On that
purpose, the non-degeneracy of the limit solutions is required. We
say that a solution $u$ of \eqref{limit} is degenerate if  the
problem:
%
\begin{multline}
 -\phi''+ \omega \phi  +
\frac 1 4 \left( \ir u^2(s)\, ds \right)^2 \phi
\\ +
\left( \ir u^2(s)\, ds \right)\left( \dis \ir u(s)
\phi(s) \, ds \right) u -pu^{p-1}\phi=0 \label{ND}
\end{multline}
has a solution $\phi \in H^1(\R)$ different from $c\, u'(r)$, $c
\in \R$.

Next proposition is devoted to this question.

\begin{proposition} \label{nd0}
Let $\o_1$, $k_0$, $k_1$, $k_2$ as in Proposition \ref{explicit}.
Then:
\begin{enumerate}

\item If $\o=\o_1$, $w_{k_0}$ is a degenerate solution of
\eqref{limit}.

\item If $\o \in (0, \o_1)$, both $w_{k_1}$ and $w_{k_2}$ are
non-degenerate solutions of \eqref{limit}.

\end{enumerate}

\end{proposition}

\begin{proof}

Take $w_k$ a solution of \eqref{limit}, and $\phi$ a solution of
\eqref{ND}. Define $$\a= \left( \ir w_k^2(s)\, ds \right)\left( \ir
w_k(s) \phi(s) \, ds \right)\in \R.$$ Then, $\phi$ is a solution of
the problem:
$$-\phi''+ \omega \phi  + \frac 1 4 \left( \ir w_k^2(s)\, ds
\right)^2 \phi -pw_k^{p-1}\phi=-\a w_k. $$

Write $\phi(r)= \psi(r) w_k'(r)$: then,
$$ \psi''(r) =- 2 \psi'(r) \frac{w_k''(r)}{w_k'(r)} + \a \frac{w_k(r)}{w_k'(r)}.  $$
This is a linear equation of first order in $\psi'(r)$, which can be
solved by variation of constants:
$$\psi'(r)= \frac{C}{(w_k'(r))^2} + \frac{\a}{2}
\frac{w_k^2(r)}{(w_k'(r))^2},\ C \in \R.$$

Observe that if $C\neq 0$, then $\phi(r)= \psi(r) w_k'(r)$ tends to
$+\infty$ at infinity, which is not possible if $\phi \in
H^1(\R)$. So, $C$ must be equal to $0$, and we are led to:
$$\psi'(r)= \frac{\a}{2} \frac{w_k^2(r)}{(w_k'(r))^2}.$$

By using \eqref{hamiltonian}, the above equation can be explicitly
integrated:  $$\psi(r)= \frac{\a}{2k} \left ( \frac{2}{(p-1)}
\frac{w_k(r)}{w_k'(r)} +  r \right) +\beta, \quad \hbox{ with }\b\in \R.$$
Therefore,
$$\phi(r)= \frac{\a}{2k} \left ( \frac{2}{p-1} w_k(r) +  r w_k'(r)\right) + \beta w_k'(r).$$

Clearly, $w_k$ is degenerate if and only if $\a$ can take values
different from zero. Take $\b=0$, and recall now the definition of
$\a$:
$$ \a = \frac{\a}{2k} \left( \ir w_k^2(s)\, ds \right)\left(
\ir \frac{2}{p-1} w_k^2(s) +  s w_k'(s) w_k(s) \, ds \right).$$

By integration by parts,
$$ \ir s w_k'(s) w_k(s) \, ds = -\frac{1}{2} \ir w_k^2(s)\, ds = -\frac{1}{2} k^{\frac{5-p}{2(p-1)}} m. $$

Then, either $\a = 0$ or
$$1= m^2 k^{\frac{6-2p}{p-1}}\frac{5-p}{4(p-1)}.$$
This, together with \eqref{eq-k}, implies that $\o= \o_1$ and $k=k_0$.

\end{proof}

Observe that, due to the nonlocal character of our problem, the
linearized problem \eqref{ND} does not have the form of a
Schr\"{o}dinger operator. Then, some more work is needed in order to
accomplish a Lyapunov-Schmidt reduction. Let us define the
following operators:
\begin{align*}
T[\phi]&= -\phi'' + \left [ \omega + \frac 1 4 \left (\ir u^2(s) \, ds \right )^2 - p u^{p-1}\right ] \phi,  \nonumber
\\
K[\phi]&= \left (\ir u^2(s) \, ds \right ) \left (\ir u(s)\phi(s) \, ds
\right )u , \nonumber
\\
L[\phi] &= T[\phi] + K[\phi].
\end{align*}

Observe that $L$, $T$, $K:H^1(\R) \cap \{u'\}^{\perp} \to
H^{-1}(\R)\cap \{u'\}^\perp.$ Of course, here orthogonality is
understood for the scalar product $\langle \cdot
\rangle_{H^1(\R)}$.

\begin{corollary} \label{nd1} Assume $\o \in (0,\o_1)$. Then the operator $L$ defined above is a
bijection for $u=w_{k_1}$ or $u=w_{k_2}$. Moreover, if $\o < \o_0$
and $u= w_{k_2}$, there exists $c>0$ so that:

\beq \label{minimoND} \langle L[\phi], \phi \rangle_{H^1(\R)} \ge c \|
\phi\|_{H^1(\R)}^2 \ \mbox{ for any } \phi \in H^1(\R) \cap
\{u'\}^{\perp}.\eeq

\end{corollary}

\begin{proof}

It is well-known that $T$ is a bijection. Since the image of $K$
is of dimension 1, $K$ is obviously a compact operator. Moreover,
by Proposition \ref{nd0}, $L$ is injective. By Fredholm
alternative, $L$ is a bijection.

Moreover, by Proposition \ref{explicit}, we know that $w_{k_2}$ is
a minimizer of $J$ with non-degenerate second derivative. From
this, \eqref{minimoND} follows.

\end{proof}

\section{The auxiliary equation}

In this section we begin the proof of Theorem \ref{main}, which
uses a singular perturbation argument based on a Lyapunov-Schmidt
reduction. So, let us assume $p\in (1,3)$ and $\o \in (0, \o_0)$
throughout the rest of the paper.

The construction of our approximating solution is largely inspired
by \cite{amni}. We set $U=w_{k_2}$, where $k_2$ is the second root of equation \eqref{eq-k} (see Proposition \ref{explicit}). We will denote $U_\rho=U(\cdot -\rho)$, and $V_\rho=U_\rho (R)e^{\sqrt{k_2}(\cdot-R)}$. For later use, observe
that $U_\rho$ solves the equation \beq\label{urho} -U_\rho''+\o
U_\rho +\frac 14 \left( \ir U^2(s) \ ds\right)^2 U_\rho =U_\rho^p,
\quad \hbox{in }\R, \eeq whereas $V_\rho$ satisfies
\beq\label{vrho} -V_\rho''+\o V_\rho +\frac 14 \left( \ir U^2(s) \
ds\right)^2 V_\rho =0, \quad \hbox{in }\R. \eeq We define our
approximating solution as $z_\rho=\phi_R (U_\rho-V_\rho)\in H_R$,
where $\phi_R:[0,+\infty)\to [0,1]$ is a smooth function such that
\[
\phi_R(r)
=\left\{
\begin{array}{ll}
0&\hbox{if }r\in (0,R/4),
\\
1&\hbox{if }r\in (R/2,R),
\end{array}
\right., \quad |\phi_R'(r)|\le \frac C R.
\]

Let us define the interval \[
\I_R=\left[R-\frac{\b}{2\sqrt{k_2}}\log R, \
R-\frac{\a}{2\sqrt{k_2}}\log R \right], \] where $\a$, $\b$ are
fixed constants satisfying:
\beq \label{alphabeta} \max \left \{\frac 1 2,\ \frac 1 p \right
\} < \a < 1,\ \ 1 < \b < \min\{2,\ p\} \a. \eeq

We are interested in finding critical points for $I$ near the
manifold of approximating solutions:
$$Z= \{ z_{\rho}, \ \rho \in \I_R \}$$

For any $z_\rho \in Z$, consider $T_{z_\rho} Z$ the tangent
space, spanned by the function:
\beq \label{zpunto} \dot{z}_\rho:= \frac{\partial z_\rho}{\partial
\rho}= \phi_R \left (-U'_\rho +  U'_\rho
(R)e^{\sqrt{k_2}(\cdot-R)}\right). \eeq

Take also $W=W_{z_\rho }= T_{z_\rho }Z^{\perp}$, and let $P, Q$ denote the
orthogonal projections onto the spaces $W$ and $T_{z_\rho }Z$,
respectively. We then decompose the equation:
$$I_R'(z_\rho +w)=0 \Leftrightarrow \left \{ \begin{array}{cl} P I_R'(z_\rho +w)=0 & \mbox{ (auxiliary equation)},
\\QI_R'(z_\rho +w)=0 & \mbox{ (bifurcation equation)}. \end{array} \right. $$

In this section we use the Banach contraction principle to find a solution
$w_\rho$ of the auxiliary equation for any $\rho \in \I_R$.
In order to do that, some estimates are in order.

Our first result is an estimate on $I_R'(z_\rho)$, which measures in
which sense $z_\rho$ is an approximating solution.

\begin{lemma}\label{pr:I'}
For any $\rho\in \I_R$, we have that
\[
\|I_R'(z_\rho)\| =\left\{
\begin{array}{ll}
O\left(\sqrt{\rho}\,|R-\rho|\,e^{-2\sqrt{k_2}(R-\rho)}\right), &
\hbox{if }p\ge 2,
\\
O\left(\sqrt{\rho}\,e^{-p\sqrt{k_2}(R-\rho)}\right), & \hbox{if
}p<2.
\end{array}
\right.
\]
\end{lemma}

\begin{proof}
Let us recall that, by \eqref{versione 2},
\begin{align*}
\frac{I_R'(z_\rho)[u]}{2\pi}
&= \int_0^R (z'_\rho u'+\o z_\rho u)r \, dr +\int_0^R
\frac{z_\rho(r)u(r)}{r}\left(\int_0^r \frac{s}{2}z_\rho^2(s) \,
ds\right)^2 dr
\\
&\quad+ \int_0^R z_\rho(r) u(r) r \left(\int_r^R
\frac{z_\rho^2(s)}{s} \left(\int_0^s \frac{t}{2} z_\rho^2(t) \,
dt\right) ds \right) dr - \int_0^R z_\rho^{p} u r\, dr.
\end{align*}
Let us evaluate all these integrals in $[0,R/2]$. We have
\[
\left| \int_0^{\frac{R}{2}} z_\rho'u' r\, dr\right|
\le C R^\frac{1}{2} \left(
\int_{\frac{R}{4}}^\frac{R}{2}e^{2\sqrt{k_2}(r-\rho)}\, dr
\right)^\frac{1}{2}\|u\| \le C R^\frac{1}{2}
e^{-\sqrt{k_2}\frac{R}{4}}\|u\|.
\]
Analogously,
\[
\left| \int_0^{\frac{R}{2}} \o z_\rho u r\, dr\right| \le C
R^\frac{1}{2} e^{-\sqrt{k_2}\frac{R}{4}}\|u\|.
\]
Moreover
\begin{align*}
\left|  \int_0^\frac R2  \frac{z_\rho(r)u(r)}{r}\left(\int_0^r
\frac{s}{2}z_\rho^2(s) \, ds\right)^2 dr  \right| &\le C \left(
\int_\frac{R}{4}^\frac{R}{2} R e^{2\sqrt{k_2}(s-\rho)}\,
ds\right)^2 \int_\frac{R}{4}^\frac{R}{2}
\left|\frac{z_\rho(r)u(r)}{r}\right|dr
\\
&\le C R^2e^{-\sqrt{k_2}\frac R2}\left(
\int_\frac{R}{4}^\frac{R}{2}
\frac{z_\rho^2(r)}{r^3}dr\right)^\frac{1}{2}\|u\|
\\
&\le C R^\frac{1}{2} e^{-\sqrt{k_2}\frac{3R}{4}}\|u\|.
\end{align*}
We can also estimate
\begin{align*}
\left|\int_0^\frac R2 z_\rho(r) u(r) r \left(\int_r^R
\frac{z_\rho^2(s)}{s} \left(\int_0^s \frac{t}{2} z_\rho^2(t) \,
dt\right) ds \right) dr\right| &\le C \left(\int_\frac{R}{4}^R
z_\rho^2(s) ds\right)^2
\int_\frac{R}{4}^\frac{R}{2}|z_\rho(r)u(r)r|dr
\\
& \le C R e^{-\sqrt{k_2}\frac{R}{4}}\|u\|.
\end{align*}
Finally
\[
\left| \int_0^{\frac{R}{2}} z_\rho^p u r\, dr\right| \le C
R^\frac{p}{p+1} e^{-p\sqrt{k_2}\frac{R}{4}}\|u\|.
\]
Therefore
\begin{align*}
\frac{I_R'(z_\rho)[u]}{2\pi} &= \int_\frac{R}{2}^R (z'_\rho u'+\o
z_\rho u)r \, dr +\int_\frac{R}2^R
\frac{z_\rho(r)u(r)}{r}\left(\int_0^r \frac{s}{2}z_\rho^2(s) \,
ds\right)^2 dr
\\
&\quad+ \int_\frac{R}{2}^R z_\rho(r) u(r) r \left(\int_r^R
\frac{z_\rho^2(s)}{s} \left(\int_0^s \frac{t}{2} z_\rho^2(t) \,
dt\right) ds \right) dr - \int_\frac{R}{2}^R z_\rho^{p} u r\, dr
\\
&\quad+o(e^{-\sqrt{k_2}\frac{R}{5}})\|u\|.
\end{align*}
Let us prove that
\begin{align}
\frac{ I_R'(z_\rho)[u]}{2\pi} &= \sqrt{\rho}\Bigg[\int_\frac{R}{2}^R
(z'_\rho u'+\o z_\rho u)\sqrt r \, dr +\frac{1}{4}\int_\frac{R}2^R
z_\rho(r)u(r)\sqrt{r}\left(\int_\frac{R}{2}^r z_\rho^2(s) \,
ds\right)^2 dr \nonumber
\\
&\quad+ \frac{1}{2}\int_\frac{R}{2}^R z_\rho(r) u(r) \sqrt r
\left(\int_r^R z_\rho^2(s) \left(\int_\frac{R}{2}^s  z_\rho^2(t)
\, dt\right) ds \right) dr - \int_\frac{R}{2}^R z_\rho^{p} u
\sqrt{r}\, dr \Bigg] \nonumber
\\
&\quad+O(R^{-1/2})\|u\|. \label{eq:sqrtrho}
\end{align}
Indeed we have
\begin{align*}
\left| \int_\frac{R}{2}^R z'_\rho u' r\, dr- \int_\frac{R}{2}^R
z'_\rho u' \sqrt{\rho}\sqrt{r}\, dr\right| &\le \left(
\int_\frac{R}{2}^R|z'_\rho|^2 |\sqrt{r}-\sqrt{\rho}|^2
dr\right)^\frac{1}{2}\|u\|
\\
&= \left( \int_\frac{R}{2}^R|z'_\rho|^2
\left|\frac{r-\rho}{\sqrt{r}+\sqrt{\rho}}\right|^2dr\right)^\frac{1}{2}\|u\|
\le \frac{C}{\sqrt{R}}\|u\|.
\end{align*}
Therefore \beq\label{eq:sqrtrho1} \int_\frac{R}{2}^R z'_\rho u'
r\, dr= \sqrt{\rho}\int_\frac{R}{2}^R z'_\rho u' \sqrt{r}\, dr
+O(R^{-1/2})\|u\|. \eeq Analogously, we get
\begin{align}
\int_\frac{R}{2}^R z_\rho u r\, dr &=\sqrt{\rho} \int_\frac{R}{2}^R z_\rho u
\sqrt{r}\, dr +O(R^{-1/2})\|u\|, \label{eq:sqrtrho2}
\\
\int_\frac{R}{2}^R z^{p}_\rho u r\, dr &= \sqrt{\rho}\int_\frac{R}{2}^R
z^{p}_\rho u' \sqrt{r}\, dr +O(R^{-1/2})\|u\|.
\label{eq:sqrtrho3}
\end{align}
Let us study the nonlocal terms;
\begin{align*}
&\int_\frac{R}2^R \frac{z_\rho(r)u(r)}{r}\left(\int_0^r
\frac{s}{2}z_\rho^2(s) \, ds\right)^2 dr -\int_\frac{R}2^R
z_\rho(r)u(r)\sqrt{\rho}\sqrt{r}\left(\int_0^r
\frac{1}{2}z_\rho^2(s) \, ds\right)^2 dr
\\
&\qquad =\underset{(I)}{\underbrace{\int_\frac{R}2^R
z_\rho(r)u(r)\left(\frac{1}{r}-
\frac{1}{\rho}\right)\left(\int_0^r \frac{s}{2}z_\rho^2(s) \,
ds\right)^2 dr}}
\\
&\qquad\quad+\underset{(II)}{\underbrace{\frac{1}{\rho}
\int_\frac{R}2^R z_\rho(r)u(r)\left[\left(\int_0^r
\frac{s}{2}z_\rho^2(s) \, ds\right)^2 -\left(\int_0^r
\frac{\rho}{2}z_\rho^2(s) \, ds\right)^2\right] dr}}
\\
&\qquad\quad+\underset{(III)}{\underbrace{\sqrt{\rho}\int_\frac{R}2^R
z_\rho(r)u(r)(\sqrt{\rho}-\sqrt{r})\left(\int_0^r
\frac{1}{2}z_\rho^2(s) \, ds\right)^2 dr}}.
\end{align*}
We have
\begin{align*}
|(I)| &\le \int_\frac{R}2^R |z_\rho(r)u(r)|\left|\frac{r-\rho}{r
\rho}\right|\left(\int_0^r \frac{s}{2}z_\rho^2(s) \, ds\right)^2
dr
\\
&\le C\int_\frac{R}2^R |z_\rho(r)u(r)|\left|r-\rho\right| dr
\\
&\le C \left(\int_\frac{R}2^R z^2_\rho(r)\frac{(r-\rho)^2}r
dr\right)^\frac{1}{2}\|u\| \le \frac{C}{\sqrt{R}}\|u\|.
\end{align*}
Moreover
\begin{align*}
|(II)| &\le \frac{1}{\rho} \int_\frac{R}2^R |z_\rho(r)u(r)|
\left|\int_0^r \frac{s-\rho}{2}z_\rho^2(s) \, ds\right|
\left|\int_0^r \frac{s+\rho}{2}z_\rho^2(s) \, ds\right| dr
\\
&\le C \int_\frac{R}2^R |z_\rho(r)u(r)| dr \le
\frac{C}{\sqrt{R}}\|u\|,
\end{align*}
and
\begin{align*}
|(III)| &\le \sqrt{\rho} \int_\frac{R}2^R |z_\rho(r)u(r)|
|\sqrt{\rho}-\sqrt{r}|\left(\int_0^r \frac{1}{2}z_\rho^2(s) \,
ds\right)^2 dr
\\
&\le  C\sqrt{\rho}\int_\frac{R}2^R |z_\rho(r)u(r)|
\frac{|\rho-r|}{\sqrt{\rho}+\sqrt{r}} dr
\le\frac{C}{\sqrt{R}}\|u\|.
\end{align*}
Therefore, taking also in account the exponential decay of
$z_\rho$,
\begin{align}
\int_\frac{R}2^R \frac{z_\rho(r)u(r)}{r}\left(\int_0^r
\frac{s}{2}z_\rho^2(s) \, ds\right)^2 dr
&=\sqrt{\rho}\int_\frac{R}2^R
z_\rho(r)u(r)\sqrt{r}\left(\int_\frac{R}{2}^r
\frac{1}{2}z_\rho^2(s) \, ds\right)^2 dr \nonumber
\\
&\quad+O(R^{-1/2})\|u\|. \label{eq:sqrtrho4}
\end{align}
Analogously
\begin{align*}
&\int_\frac{R}{2}^R z_\rho(r) u(r) r \left(\int_r^R
\frac{z_\rho^2(s)}{s} \left(\int_0^s \frac{t}{2} z_\rho^2(t) \,
dt\right) ds \right) dr
\\
&\qquad\quad-\int_\frac{R}{2}^R z_\rho(r) u(r) \sqrt{\rho}
\sqrt{r} \left(\int_r^R z_\rho^2(s) \left(\int_0^s \frac{1}{2}
z_\rho^2(t) \, dt\right) ds \right) dr
\\
&\qquad=\int_\frac{R}{2}^R z_\rho(r) u(r) (r-\rho) \left(\int_r^R
\frac{z_\rho^2(s)}{s} \left(\int_0^s \frac{t}{2} z_\rho^2(t) \,
dt\right) ds \right) dr
\\
&\qquad\quad +\int_\frac{R}{2}^R z_\rho(r) u(r) \left(\int_r^R
z_\rho^2(s)\left(\frac{\rho}{s}-1\right) \left(\int_0^s
\frac{t}{2} z_\rho^2(t) \, dt\right) ds \right) dr
\\
&\qquad\quad +\int_\frac{R}{2}^R z_\rho(r) u(r)  \left(\int_r^R
z_\rho^2(s) \left(\int_0^s \left(\frac{t}{2}-\frac{\rho}{2}\right)
z_\rho^2(t) \, dt\right) ds \right) dr
\\
&\qquad\quad +\int_\frac{R}{2}^R z_\rho(r) u(r)
\sqrt{\rho}(\sqrt{\rho}- \sqrt{r}) \left(\int_r^R z_\rho^2(s)
\left(\int_0^s \frac{1}{2} z_\rho^2(t) \, dt\right) ds \right) dr
\\
&\qquad=O(R^{-1/2})\|u\|.
\end{align*}
and so, again by the exponential decay of $z_\rho$,
\begin{align}
&\int_\frac{R}{2}^R z_\rho(r) u(r) r \left(\int_r^R
\frac{z_\rho^2(s)}{s} \left(\int_0^s \frac{t}{2} z_\rho^2(t) \,
dt\right) ds \right) dr \nonumber
\\
&\qquad=\sqrt{\rho} \int_\frac{R}{2}^R z_\rho(r) u(r) \sqrt{r}
\left(\int_r^R z_\rho^2(s) \left(\int_\frac{R}{2}^s \frac{1}{2}
z_\rho^2(t) \, dt\right) ds \right) dr \nonumber
\\
&\qquad\quad+O(R^{-1/2})\|u\|. \label{eq:sqrtrho5}
\end{align}
Then, by \eqref{eq:sqrtrho1}, \eqref{eq:sqrtrho2},
\eqref{eq:sqrtrho3}, \eqref{eq:sqrtrho4} and \eqref{eq:sqrtrho5},
we get \eqref{eq:sqrtrho}.
\\
Let us observe that
\begin{align*}
\sqrt{\rho} \int_\frac{R}{2}^R z_\rho'u' \sqrt{r} \, dr
&=-\sqrt{\rho} \int_\frac{R}{2}^R  z_\rho'' u \sqrt{r} \, dr
-\sqrt{\rho} \int_\frac{R}{2}^R \frac{z_\rho'}{2\sqrt{r}}  u \, dr
+o(R^{-1/2})\|u\|
\\
&=-\sqrt{\rho} \int_\frac{R}{2}^R  z_\rho'' u \sqrt{r} \, dr
+O(R^{-1/2})\|u\|.
\end{align*}
Therefore, by \eqref{urho} and \eqref{vrho}, we have
\begin{align*}
 \frac{I_R'(z_\rho)[u]}{2\pi} &= \sqrt{\rho}\int_\frac{R}{2}^R
\left[-z''_\rho +\o z_\rho
+\frac{z_\rho}{4}\Bigg[\left(\int_\frac{R}{2}^r z_\rho^2(s) \,
ds\right)^2 + 2 \left(\int_r^R z_\rho^2(s)
\left(\int_\frac{R}{2}^s  z_\rho^2(t) \, dt\right) ds
\right)\Bigg] -z_\rho^{p}\right] u \sqrt{r}\, dr
\\
&\quad+O(R^{-1/2})\|u\|
\\
&= \sqrt{\rho}\int_\frac{R}{2}^R \Bigg[-z''_\rho +\o z_\rho
+\frac{z_\rho}{4}\left(\int_\frac{R}{2}^R z_\rho^2(s) \,
ds\right)^2 -z_\rho^{p}\Bigg] u \sqrt{r}\, dr
\\
&\quad+O(R^{-1/2})\|u\|
\\
&= \underset{(A)}{\underbrace{\frac{\sqrt{\rho}}4\Bigg[
\left(\int_{\frac{R}{2}}^{R}U_\rho^2(s)\, ds\right)^2 -
\left(\int_{-\infty}^{+\infty}U^2(s)\, ds\right)^2 \Bigg]
\int_\frac{R}{2}^R  U_\rho u \sqrt{r}\, dr }}
\\
&\quad-\underset{(B)}{\underbrace{\frac{\sqrt{\rho}}4 \Bigg[
\left(\int_{\frac{R}{2}}^{R}U_\rho^2(s)\, ds\right)^2 -
\left(\int_{-\infty}^{+\infty}U^2(s)\, ds\right)^2 \Bigg]
\int_\frac{R}{2}^R V_\rho u \sqrt{r}\, dr }}
\\
&\quad-\underset{(C)}{\underbrace{\sqrt{\rho}
\left(\int_{\frac{R}{2}}^{R}U_\rho^2(s)\, ds\right)
\left(\int_{\frac{R}{2}}^{R}U_\rho(s) V_\rho(s)\, ds\right)
\int_\frac{R}{2}^R U_\rho u \sqrt{r}\, dr}}
\\
&\quad+\underset{(D)}{\underbrace{\frac{\sqrt{\rho}}2
\left(\int_{\frac{R}{2}}^{R}U_\rho^2(s)\, ds\right)
 \left(\int_{\frac{R}{2}}^{R}V_\rho^2(s)\, ds\right)
\int_\frac{R}{2}^R U_\rho u \sqrt{r}\, dr}}
\\
&\quad-\underset{(E)}{\underbrace{\sqrt{\rho}\int_\frac{R}{2}^R\left[(U_\rho-V_\rho)^{p}-U_\rho^p\right]u
\sqrt{r}\, dr }}
\\
&\quad+o\left(e^{-2\sqrt{k_2}(R-\rho)}\right)\|u\|.
\end{align*}
Let us evaluate every term in the previous formula. Since
\[
\left(\int_{\frac{R}{2}}^{R}U_\rho^2(s)\, ds\right)^2 -
\left(\int_{-\infty}^{+\infty}U^2(s)\, ds\right)^2
=O\left(e^{-2\sqrt{k_2}(R-\rho)}\right),
\]
we infer that
\begin{align*}
(A)&=O\left(\sqrt{\rho}\,e^{-2\sqrt{k_2}(R-\rho)}\right)\|u\|,
\\
(B)&=o\left(\sqrt{\rho}\,e^{-2\sqrt{k_2}(R-\rho)}\right)\|u\|.
\end{align*}

Moreover, by \cite[Lemma 3.7]{ACR},
\[
(C)=O\left(\sqrt{\rho}\,|R-\rho|\,e^{-2\sqrt{k_2}(R-\rho)}\right)\|u\|.
\]

The estimate of $(D)$ is direct:
\begin{align*}
\\
(D)&=O\left(\sqrt{\rho}\,e^{-2\sqrt{k_2}(R-\rho)}\right)\|u\|.
\end{align*}

Finally, by the Mean Value Theorem we can estimate:
\[
|(U_\rho-V_\rho)^{p}-U_\rho^p| \le C  U_\rho^{p-1}V_\rho.
\]
Again by \cite[Lemma 3.7]{ACR}, we have
\begin{align*}
(E)&\le
\sqrt{\rho}\left(\int_\frac{R}{2}^R\left[(U_\rho-V_\rho)^{p}-U_\rho^p\right]^2
dr\right)^\frac{1}{2}\|u\|
\le \sqrt{\rho}\left(\int_\frac{R}{2}^R U_\rho^{2(p-1)}V_\rho^2\,
dr\right)^\frac{1}{2}\|u\|
\\
&= \left\{
\begin{array}{ll}
O\left(\sqrt{\rho}\,e^{-2\sqrt{k_2}(R-\rho)}\right)\|u\|, &
\hbox{if }p>2,
\\
O\left(\sqrt{\rho}\,|R-\rho|^\frac{1}{2}\,e^{-2\sqrt{k_2}(R-\rho)}\right)\|u\|,
& \hbox{if }p=2,
\\
O\left(\sqrt{\rho}\,e^{-p\sqrt{k_2}(R-\rho)}\right)\|u\|, &
\hbox{if }p<2.
\end{array}
\right.
\end{align*}
Hence we get the conclusion.
\end{proof}

Our next result concerns the non-degeneracy of $I_R''(z_\rho)$, and
makes use of Corollary \ref{nd1} in an essential way.

\begin{lemma} \label{3.2}
There exist $C>c>0$ such that for large $R$ and for any $\rho \in \I_R$,
\[
c \|u\|^2 \le I_R''(z_\rho)[u,u]\le C \|u\|^2, \quad\hbox{ for all }u \perp
T_{z_\rho}Z.
\]
\end{lemma}

\begin{proof}
Let $u\in H_R$  such that $u \perp T_{z_\rho}Z$. Take $a\in
[R-2\sqrt{R},R-\sqrt{R}]$ such that
\[
\int_a^{a+1}\left(|u'|^2+\o u^2\right)r\, dr\le \frac{2}{\sqrt{R}}
\|u\|^2.
\]
Take $\psi$ a cut-off function such that
\[
\psi(r)=\left\{
\begin{array}{ll}
0&\hbox{if }0 \le r\le a,
\\
(r-a) &\hbox{if }a \le r\le a+1,
\\
1&\hbox{if } r\ge a+1.
\end{array}
\right.
\]
Since $u=\psi u+(1-\psi)u$, we have
\[
\|u\|^2=\|\psi u \|^2+\|(1-\psi) u \|^2+2\langle \psi u ,
(1-\psi)u\rangle.
\]
Moreover
\[
I_R''(z_\rho)[u,u] =I_R''(z_\rho)[\psi u,\psi u] +I_R''(z_\rho)[(1-\psi)
u,(1-\psi) u] +2I_R''(z_\rho)[\psi u,(1-\psi) u].
\]
For that sake of brevity, we set $v=\psi u$ and $w=(1-\psi)u$. By
\eqref{eq:I''}, we infer that
\begin{align}
I_R''(z_\rho)[v,w]&= \langle v,w\rangle + o(1)\|u\|=
o(1)\|u\|,\label{eq:I''vw}
\\
I_R''(z_\rho)[w,w]&= \|w\|^2 + o(1)\|u\|.\label{eq:I''ww}
\end{align}
Before evaluating $I_R''(z_\rho)[v,v]$, let us observe that since
$u \perp T_{z_\rho}Z$, then \beq \label{zeroessima} |\langle
v,\dot{z}_\rho\rangle|=|-\langle
w,\dot{z}_\rho\rangle|=o\left(e^{-R^{1/3}}\right)\|u\|. \eeq We
now extend $v$ to $0$ in all $\R$, and estimate $ \langle
v,U'_\rho \rangle_{H^1(\R)}.$ On that purpose, it is easy to see
that: \beq \label{prima} \rho \langle v,U'_\rho \rangle_{H^1(\R)}
- \langle v,U'_\rho \rangle = O(R^{-1/2}) \| v\|. \eeq

Moreover, by using \eqref{zpunto},
\beq \label{seconda}   \langle v,\dot{z}_\rho + U'_\rho
\rangle = U'(R-\rho) \langle v, e^{\sqrt{k_2}(\cdot - R)}
\rangle = |U'(R-\rho)| O(\sqrt{R})\|v\|. \eeq

Putting together \eqref{zeroessima}, \eqref{prima} and
\eqref{seconda}, we conclude that \beq \label{pofale} \langle v,
U'_\rho \rangle_{H^1(\R)} = o(R^{-1/2}) \|v\|,\eeq

We now study $I_R''(z_\rho)[v,v]$; here we will derive $I_R''$ from
expression \eqref{versione 2}, obtaining the expression:
\begin{align}
\frac{I_R''(z_\rho)[u,v]}{2\pi} &=\int_0^R (u' v'+\o uv)r \, dr
+\int_0^R \frac{u(r)v(r)}r\left(\int_0^r \frac{s}{2}z_\rho^2(s) \,
ds \right)^2 dr  \nonumber
\\
&\quad+\int_0^R \frac{z_\rho(r)u(r)}r\left(\int_0^r s z_\rho^2(s)
\, ds \right)\left(\int_0^r s z_\rho(s) v(s)\, ds \right) dr
\nonumber
\\
&\quad+\int_0^R u(r)v(r)r \left( \int_r^R \frac{z_\rho^2(s)}{s}
\left( \int_0^s \frac t2 z_\rho^2(t)\, dt  \right) ds \right)dr
\nonumber
\\
&\quad+\int_0^R z_\rho(r)u(r)r \left( \int_r^R
\frac{z_\rho(s)v(s)}{s}  \left( \int_0^s t z_\rho^2(t)\, dt
\right) ds \right)dr \nonumber
\\
&\quad+\int_0^R z_\rho(r)u(r)r \left( \int_r^R
\frac{z_\rho^2(s)}{s}  \left( \int_0^s t z_\rho(t)v(t)\, dt
\right) ds \right)dr \nonumber
\\
&\quad-p\int_0^R  z_\rho^{p-1} u v r \,dr. \label{eq:I''}
\end{align}

Arguing as in the proof of Lemma \ref{pr:I'}, we can substitute
$r,s,t$ with $\rho$ committing negligible errors: more precisely,
we have
\begin{align*}
\frac{I_R''(z_\rho)[v,v]}{2\pi}
&=\rho \int_0^R \left((v')^2+\o v^2\right) dr +\frac{\rho}4
\int_0^R v^2(r)\left(\int_0^r z_\rho^2(s) \, ds \right)^2 dr
\\
&\quad+\rho \int_0^R z_\rho(r)v(r)\left(\int_0^r  z_\rho^2(s) \,
ds \right)\left(\int_0^r  z_\rho(s) v(s)\, ds \right) dr
\\
&\quad+\frac \rho 2 \int_0^R v^2(r) \left( \int_r^R z_\rho^2(s)
\left( \int_0^s  z_\rho^2(t)\, dt  \right) ds \right)dr
\\
&\quad+\rho \int_0^R z_\rho(r)v(r) \left( \int_r^R z_\rho(s)v(s)
\left( \int_0^s  z_\rho^2(t)\, dt  \right) ds \right)dr
\\
&\quad+\rho \int_0^R z_\rho(r)v(r) \left( \int_r^R z_\rho^2(s)
\left( \int_0^s  z_\rho(t)v(t)\, dt  \right) ds \right)dr
\\
&\quad-\rho p \int_0^R  z_\rho^{p-1} v^2  \,dr +o(1)\|u\|^2.
\end{align*}
Let us observe that,
\begin{multline*}
\rho \int_0^R \left((v')^2+\o v^2\right) dr -\rho p\int_0^R
z_\rho^{p-1} v^2  \,dr
\\
=\rho \int_{-\infty}^{+\infty} \left((v')^2+\o v^2\right) dr -\rho
p\int_{-\infty}^{+\infty}  U_\rho^{p-1} v^2  \,dr +o(1)\|u\|^2.
\end{multline*}
Setting $F(r)=\int_0^r z_\rho^2(s) \, ds$, since
$F'(r)=z_\rho^2(r)$,  we get
\begin{align*}
&\frac{\rho}{4}\int_0^R v^2(r)\left(\int_0^r z_\rho^2(s) \, ds
\right)^2 dr +\frac \rho 2\int_0^R v^2(r) \left( \int_r^R
z_\rho^2(s)  \left( \int_0^s  z_\rho^2(t)\, dt  \right) ds
\right)dr
\\
&\qquad\qquad=\frac \rho 4 \left(\int_0^R z_\rho^2(r) \, dr
\right)^2\left(\int_0^R v^2(r) \, dr\right)
\\
&\qquad\qquad=\frac \rho 4 \left(\int_{-\infty}^{+\infty}
U_\rho^2(r) \, dr \right)^2\left(\int_{-\infty}^{+\infty} v^2(r)
\, dr\right) +o(1)\|u\|^2.
\end{align*}
Setting, moreover, $G(r)=\int_0^r  z_\rho(s) v(s)\, ds $, since
$G'(r)=z_\rho(r) v(r)$, we have
\begin{align*}
&\rho \int_0^R z_\rho(r)v(r)\left(\int_0^r  z_\rho^2(s) \, ds
\right)\left(\int_0^r  z_\rho(s) v(s)\, ds \right) dr
\\
&\qquad\qquad+\rho \int_0^R z_\rho(r)v(r) \left( \int_r^R
z_\rho(s)v(s)  \left( \int_0^s  z_\rho^2(t)\, dt  \right) ds
\right)dr
\\
&\qquad\qquad+\rho \int_0^R z_\rho(r)v(r) \left( \int_r^R
z_\rho^2(s)  \left( \int_0^s  z_\rho(t)v(t)\, dt  \right) ds
\right)dr
\\
&\qquad=\rho\left(\int_0^R  z_\rho^2(r) \, dr
\right)\left(\int_0^R  z_\rho(r) v(r)\, dr \right)^2
\\
&\qquad=\rho\left(\int_{-\infty}^{+\infty}   U_\rho^2(r) \, dr
\right)\left(\int_{-\infty}^{+\infty}   U_\rho(r) v(r)\, dr
\right)^2 +o(1)\|u\|^2.
\end{align*}
Therefore, by \eqref{pofale}:
\[
I_R''(z_\rho)[v,v]= \rho L[v,v]+o(1)\|u\|^2.
\]
By Corollary \ref{nd1} we have that $c \|v\|_{H^1(\R)}^2 \le L[v,v] \le C \|v\|_{H^1(\R)}^2$. Observe moreover that $ \rho \|v\|_{H^1(\R)}^2 = \|v\|^2 +o(1) \|v\|^2$. These estimates, together with \eqref{eq:I''vw} and \eqref{eq:I''ww},
yield the conclusion.
\end{proof}

Our next estimate implies that the previous non-degeneracy result
applies also for $I_R''(z_\rho +w)$ if $w$ is sufficiently small.

\begin{lemma} \label{3.3} There exists $C>0$ such that
\[ \left | I_R''(z_\rho)[u,u] - I_R''(z_\rho+w)[u,u] \right | \le C
(\|w\|+ \|w\|^{p-1}) \|u\|^2,\] for all $u,\ w \in H_R$, $R$ sufficiently large and $\rho \in
\I_R$.

\end{lemma}

\begin{proof}

For this estimate, we use the form of $I_R''(z_\rho )$ derived from
\eqref{versione 1}, that is:

\begin{align*}
\frac{I_R''(z_\rho )[u,u]}{2\pi} &=\int_0^R \left(u'^2+\o u^2\right)r
\, dr - p \int_0^R z_\rho ^{p-1}u^2 r\, dr
\\
&\quad+ \frac 1 2 \int_0^R
\frac{z_\rho ^2(r)}{r} \left(\int_0^r s z_\rho ^2(s)\, ds \right)
\left(\int_0^r u^2(s) s \, ds \right) dr
\\
&\quad + 2
\int_0^R \frac{z_\rho (r)u(r)}{r} \left(\int_0^r s z_\rho (s)u(s) \, ds
\right) \left(\int_0^r s z_\rho ^2(s) \, ds \right)   dr
\\
& \quad+ \frac 1 4 \int_0^R \frac{u^2(r)}{r}\left(\int_0^r s z_\rho ^2(s) \, ds
\right)^2   dr + \int_0^R \frac{z_\rho ^2(r)}{r} \left(\int_0^r s
z_\rho (s)u(s) \, ds \right)^2 dr.
\end{align*}

For the estimate of the local term, we use the inequalities:
$$|a+b|^{q} - |a|^q \le \left \{ \begin{array}{ll} C
(|a|^{q-1}+|b|^{q-1})|b| & \mbox{ if } q >1, \\ \\ |b|^q & \mbox{
if } q \in (0, 1] \end{array} \right., \ \ a,\ b \ge 0.$$

Then, if $p >2$, we use the $L^\infty$ bound of $z_\rho$ to get:
\begin{multline*}
\int_0^R  (|z_\rho+ |w||^{p-1} -
|z_\rho |^{p-1})u^2 r\, dr \le C \int_0^R (|w| +
|w|^{p-1})u^2 r \, dr
\\
\le C\left (\int_0^R (|w|^2 +
|w|^{2(p-1)}) r\, dr \right )^{1/2} \left (\int_0^R u^4 r\, dr
\right )^{1/2}.
\end{multline*}
If $p \in (1,2]$,
\begin{multline*}
 \int_0^R  (|z_\rho+ |w||^{p-1} -
|z_\rho |^{p-1})u^2 r\, dr \le \int_0^R |w|^{p-1}u^2r\, dr
\\
\le \left (\int_0^R |w|^2 r\, dr \right )^{\frac{p-1}{2}}
\left (\int_0^R |u(r)|^{\frac{4}{3-p}}r\, dr \right
)^{\frac{3-p}{2}}.
\end{multline*}

\medskip Let us now consider the nonlocal terms. Observe that each
of those terms is of order 4 in $z_\rho $. The estimates of each term is
relatively easy, but the whole computation is lengthy. Let us
explain in detail some of them. For instance, let us define:
$$ Q(z_1,z_2,z_3,z_4)= \int_0^R
\frac{z_1(r)z_2(r)}{r} \left(\int_0^r s z_3(s)z_4(s)\, ds \right)
\left(\int_0^r s u^2(s) \, ds \right) \, dr. $$

Clearly $Q$ is linear in each variable. Therefore, we need to
estimate:
$$Q(z_1, z_2, z_3, z_4) \quad \mbox { where } z_i=w \ \forall \ i \in A, \quad \ z_i = z_\rho \ \forall \ i \notin A,$$
for any non-empty $A \subset \{1,2,3,4\}$. For instance, we
estimate:
\begin{align*}
 Q(z_\rho, z_\rho, z_\rho,w)&= \int_0^R
\frac{z_\rho^2}{r} \left(\int_0^r s z_\rho(s)w(s)\, ds \right)
\left(\int_0^r u^2(s)s  \, ds \right) \, dr
%
%
%
\\
 &\le \int_0^R
\frac{z_\rho^2}{\sqrt{r}} \, dr \int_0^R  z_\rho(s)w(s) \sqrt{s}
\, ds \int_0^R u^2(s) s \, ds
\\
& \le  \int_0^R
\frac{z_\rho^2}{\sqrt{r}} \, dr \cdot \left(\int_0^R  z_\rho^2 \,
dr\right)^\frac 12 \|w\|_{L^2} \|u\|^2\le \frac{C}{\sqrt{R}}
\|w\|_{L^2} \| u\|_{L^2}^2.
\end{align*}
The cases where $w$ appears more than once are easier.

Slightly more difficult is the estimate of $K(z_\rho, z_\rho,
z_\rho, w)$, with
$$K(z_1,z_2,z_3,z_4)=\int_0^R \frac{z_1(r)u(r)}{r} \left(\int_0^r s z_2(s)u(s) \, ds
\right) \left(\int_0^r s z_3(s)z_4(s) \, ds \right)  \, dr$$
Indeed, in this case $z_\rho$ appears in the three integrals.
Then, we need to share the variable, as follows:
\begin{align*}
K(z_\rho, z_\rho, z_\rho,w)&= \int_0^R \frac{z_\rho(r) u(r)}{r} \left(\int_0^r s z_\rho(s)u(s) \, ds
\right) \left(\int_0^r s z_\rho(s)w(s) \, ds \right)  \, dr
\\
&\le \int_0^R \frac{z_\rho(r)}{\sqrt{r}} u(r) \sqrt{r} \, dr \int_0^R  z_\rho(s)u(s) \sqrt{s} \, ds
\int_0^R  z_\rho(s)w(s) \sqrt{s} \, ds   \, dr
\\
& \le \frac{C}{\sqrt{R}} \|u\|_{L^2}^2 \|w\|_{L^2}.
\end{align*}

The other terms can be estimated in an analogous way.

\end{proof}

We are now in conditions of stating the main result of this
section.

\begin{proposition} \label{auxiliary} For sufficiently large $R$ and any $\rho \in \I_R$, there exists a unique
$w_\rho \in H_R$ such that:

\begin{enumerate}

\item $z_\rho + w_\rho$ solves the auxiliary
equation $P I_R'(z_\rho + w_\rho)=0$.

\item The map $\rho \mapsto w_\rho$ is $C^1$.

\item
\[
\|w_\rho\| =\left\{
\begin{array}{ll}
O\left(R^{\frac{1-2\a}{ 2}} \log R \right), & \hbox{if }p\ge 2,
\\
O\left(R^{\frac{1-p \a}{ 2}}\right), & \hbox{if }p<2.
\end{array}
\right.
\]
\end{enumerate}
\end{proposition}

\begin{proof}

The proof is quite standard in Lyapunov-Schmidt reduction. By
Lemma \ref{pr:I'}, and taking into account \eqref{alphabeta}, we
obtain that $I_R'(z_\rho)= o(1)$. This, together with the
estimates of Lemmas \ref{3.2}, \ref{3.3}, makes it possible to carry
out the contraction argument. Moreover, the $C^1$ regularity of the map $\rho \mapsto w_\rho$ can be deduced from the Implicit Function Theorem. See \cite{ambook}, Sections 2.3,
2.4, for details.

Moreover, by taking into account that $\rho \in \I_R$, we
conclude that:

\[
\|I_R'(z_\rho)\| =\left\{
\begin{array}{ll}
O\left(R^{\frac{1-2\a}{ 2}} \log R \right), & \hbox{if }p\ge 2,
\\
O\left(R^{\frac{1-p \a}{ 2}}\right), & \hbox{if }p<2.
\end{array}
\right.
\]

Therefore, the same estimate for $w_\rho$ holds.

\end{proof}

\begin{remark} It is important to observe that most of the results of this
section hold with the same proof if we consider $\o \in (0,\o_1)$,
and also if we take $U= w_{k_1}$ instead of $w_{k_2}$ (Lemma \ref{3.2} should
be restated for $U=w_{k_1}$, since this is not a minimizer for $J$). In next
section, though, we will obtain a solution of the bifurcation equation only if $J(U)<0$. This is the reason why we need $ \o \in (0, \o_0)$,
$U=w_{k_2}$, according to Proposition \ref{explicit}. But
those restrictions are natural, as commented in the Introduction.

\end{remark}

\section{The reduced functional}

Let us define:

$$\tilde{Z}=\{ z_\rho + w_\rho:\ \rho \in \I_R,\ w_\rho \mbox{ as in Proposition \ref{auxiliary}} \}.$$

It is well-known (see \cite{ambook}) that $\tilde{Z}$ is a natural constraint for $I_R$ or, in other words, that critical points of $I_R|_{\tilde{Z}}$ correspond to solutions of the bifurcation equation, and hence they are true critical points of $I_R$.

This section is devoted to prove the existence of a minimum for $I_R|_{\tilde{Z}}$, which receives the name of reduced functional.
\begin{lemma} \label{reduced}
For any $\rho\in \I_R$, we have
\[
I_R(z_\rho)=2\pi J(U)\rho +2 \pi \sqrt{k_2} \rho \ e^{-2\sqrt{k_2}(R-\rho)}
+o\left(\rho \ e^{-2\sqrt{k_2}(R-\rho)}\right).
\]
\end{lemma}

\begin{proof}
We have
\begin{align*}
\frac{I_R(z_\rho)}{2\pi} &= \frac 12\int_0^R \left(|z'_\rho|^2 +\o z_\rho^2
\right)r \, dr +\frac 18\int_0^R
\frac{z_\rho^2(r)}{r}\left(\int_0^r s z_\rho^2(s) \, ds\right)^2
dr
\\
&\quad-\frac{1}{p+1} \int_0^R z_\rho^{p+1} r\, dr.
\end{align*}
Let us evaluate all these integrals in $[0,R/2]$. We have
\[
\left| \int_0^{\frac{R}{2}} |z_\rho'|^2 r\, dr\right|
\le C R e^{2\sqrt{k_2}\left(\frac{R}{2}-\rho\right)} \le C R
e^{-\sqrt{k_2}\frac{R}{2}}.
\]
Analogously,
\begin{align*}
\int_0^{\frac{R}{2}} \o z_\rho^2 r\, dr &\le C R
e^{-\sqrt{k_2}\frac{R}{2}},
\\
\int_0^\frac R2  \frac{z_\rho^2(r)}{r}\left(\int_0^r s z_\rho^2(s)
\, ds\right)^2 dr &\le C R e^{-\sqrt{k_2}\frac{3R}{2}},
\\
\int_0^{\frac{R}{2}} z_\rho^{p+1}  r\, dr &\le C R
e^{-(p+1)\sqrt{k_2}\frac{R}{4}}.
\end{align*}
Therefore
\begin{align*}
\frac{I_R(z_\rho)}{2\pi} &= \frac 12\int_\frac{R}{2}^R \left(|z'_\rho|^2 +\o
z_\rho^2 \right)r \, dr +\frac 18\int_\frac{R}{2}^R
\frac{z_\rho^2(r)}{r}\left(\int_0^r s z_\rho^2(s) \, ds\right)^2
dr
\\
&\quad-\frac{1}{p+1} \int_\frac{R}{2}^R z_\rho^{p+1} r\, dr
+o(e^{-\sqrt{k_2}\frac{R}{5}}).
\end{align*}
Let us check that \beq\label{eq:rho} \frac{I_R(z_\rho)}{2\pi} = \rho
\left[\frac{1}{2}\int_\frac{R}{2}^R \left(|z'_\rho|^2 +\o z_\rho^2
\right) dr +\frac 1{24}\left( \int_\frac{R}{2}^R   z_\rho^2 \,
dr\right)^3 -\frac{1}{p+1} \int_\frac{R}{2}^R z_\rho^{p+1} \, dr
\right] +O(1). \eeq Indeed
\[
\left|\int_\frac{R}{2}^R|z'_\rho|^2r\,dr-\rho
\int_\frac{R}{2}^R|z'_\rho|^2\,dr\right| \le
\int_\frac{R}{2}^R|z_\rho '(r)|^2|r-\rho|\,dr \le C,
\]
and hence \beq\label{eq:rho1}
\int_\frac{R}{2}^R|z'_\rho|^2r\,dr=\rho
\int_\frac{R}{2}^R|z'_\rho|^2\,dr+O(1). \eeq Analogously
\begin{align}
\int_\frac{R}{2}^R z_\rho^2 r\,dr&=\rho \int_\frac{R}{2}^R
z_\rho^2\,dr+O(1),\label{eq:rho2}
\\
\int_\frac{R}{2}^R z_\rho^{p+1}r\,dr&=\rho \int_\frac{R}{2}^R
z_\rho^{p+1}\,dr+O(1). \label{eq:rho3}
\end{align}
Moreover
\begin{align*}
&\int_\frac{R}{2}^R  \frac{z_\rho^2(r)}{r}\left(\int_0^r s
z_\rho^2(s) \, ds\right)^2 dr -\rho\int_\frac{R}{2}^R
z_\rho^2(r)\left(\int_0^r z_\rho^2(s) \, ds\right)^2 dr
\\
&\qquad=\underset{(I)}{\underbrace{\int_\frac{R}{2}^R z_\rho^2(r)
\left(\frac{1}{r}-\frac{1}{\rho}\right)\left(\int_0^r s
z_\rho^2(s) \, ds\right)^2 dr }}
\\
&\qquad\quad+\underset{(II)}{\underbrace{\frac{1}{\rho}\int_\frac{R}{2}^R
z_\rho^2(r) \left[\left(\int_0^r s z_\rho^2(s) \, ds\right)^2
-\left(\int_0^r \rho z_\rho^2(s) \, ds\right)^2 \right]dr }}.
\end{align*}
Since
\begin{align*}
|(I)| &\le C\int_\frac{R}2^R
z_\rho^2(r)|r-\rho|\left(\int_\frac{R}{2}^R z_\rho^2(s) \,
ds\right)^2 dr \le C.
\\
|(II)| &\le \frac{C}{\rho} \int_\frac{R}2^R z_\rho(r)^2
\left|\int_0^r \frac{s-\rho}{2}z_\rho^2(s) \, ds\right|
\left|\int_0^r \frac{s+\rho}{2}z_\rho^2(s) \, ds\right| dr \le C ,
\end{align*}
we have
\begin{align}
\int_\frac{R}{2}^R  \frac{z_\rho^2(r)}{r}\left(\int_0^r s
z_\rho^2(s) \, ds\right)^2 dr
&=\rho\int_\frac{R}{2}^R  z_\rho^2(r)\left(\int_\frac{R}{2}^r
z_\rho^2(s) \, ds\right)^2 dr +O(1)\nonumber
\\
&=\frac{\rho}3\left(\int_\frac{R}{2}^R  z_\rho^2(r)\, dr\right)^3
+O(1). \label{eq:rho4}
\end{align}
Hence \eqref{eq:rho} follows by \eqref{eq:rho1}, \eqref{eq:rho2},
\eqref{eq:rho3} and \eqref{eq:rho4}.
\\
By \eqref{eq:rho}, we infer that
\begin{align*}
\frac{I_R(z_\rho)}{2\pi} &= \rho
\Bigg[\;\underset{(A)}{\underbrace{\frac{1}{2}\int_\frac{R}{2}^R
\left(|U'_\rho|^2 +\o U_\rho^2 \right) dr +\frac 1{24}\left(
\int_\frac{R}{2}^R   U_\rho^2 \, dr\right)^3 -\frac{1}{p+1}
\int_\frac{R}{2}^R U_\rho^{p+1} \, dr }}
\\
&\quad\underset{(B)}{\underbrace{-\int_\frac{R}{2}^R (U'_\rho
V'_\rho + \o U_\rho V_\rho )  dr -\frac 1{4}\left(
\int_\frac{R}{2}^R   U_\rho^2 \, dr\right)^2  \int_\frac{R}{2}^R
U_\rho V_\rho \, dr + \int_\frac{R}{2}^R U_\rho^{p}V_\rho \, dr }}
\\
&\quad\underset{(C)}{\underbrace{+\frac{1}{2}\int_\frac{R}{2}^R
\left(|V'_\rho|^2 +\o V_\rho^2 \right)  dr +\frac 1{8}\left(
\int_\frac{R}{2}^R   U_\rho^2 \, dr\right)^2  \int_\frac{R}{2}^R
V_\rho^2 \, dr}}
\\
&\quad\underset{(D)}{\underbrace{-\frac{1}{p+1} \int_\frac{R}{2}^R
\left( (U_\rho-V_\rho)^{p+1} - U_\rho^{p+1} +(p+1)U_\rho^{p}
V_\rho\right) dr}} \;\Bigg]
\\
&\quad+O\left(\rho|R-\rho|^2 e^{-4\sqrt{k_2}(R-\rho)}\right).
\end{align*}
We have to evaluate each term of the previous formula. Firstly,
let us observe that
\[
\int_\frac{R}{2}^R   U_\rho^2 \, dr =\int_{-\infty}^{+\infty} U^2
\, dr -\frac{e^{-2\sqrt{k_2}(R-\rho)}}{2\sqrt{k_2}}
+o(e^{-2\sqrt{k_2}(R-\rho)}).
\]
Since $U_\rho$ is a solution of \eqref{urho}, we have
\begin{align*}
(A)&=\frac{1}{2}\int_\frac{R}{2}^R \left(-U''_\rho +\o U_\rho
\right)U_\rho\, dr +\frac 1{24}\left( \int_\frac{R}{2}^R U_\rho^2
\, dr\right)^3 -\frac{1}{p+1} \int_\frac{R}{2}^R U_\rho^{p+1} \,
dr
\\
&\quad+\frac 12 U_\rho(R)U'_\rho(R) +o\left(e^{-\sqrt{k_2}\frac
R4}\right)
\\
&=-\frac{1}{8}\left(\int_{-\infty}^{+\infty}   U^2 \, dr\right)^2
\int_\frac{R}{2}^R   U_\rho^2 \, dr +\frac 1{24}\left(
\int_\frac{R}{2}^R   U_\rho^2 \, dr\right)^3 +\left(\frac
12-\frac{1}{p+1}\right) \int_\frac{R}{2}^R U_\rho^{p+1} \, dr
\\
&\quad- \frac 12\sqrt{k_2} e^{-2\sqrt{k_2}(R-\rho)}
+o\left(e^{-\sqrt{k_2}\frac R4}\right)
\\
&=-\frac{1}{12}\left(\int_{-\infty}^{+\infty}   U^2 \, dr\right)^3
+\left(\frac 12-\frac{1}{p+1}\right) \int_{-\infty}^{+\infty}
U^{p+1} \, dr -\frac 12\sqrt{k_2} e^{-2\sqrt{k_2}(R-\rho)}
\\
&\quad +o\left(e^{-2\sqrt{k_2}(R-\rho)}\right)
\\
&=J(U) -\frac 12\sqrt{k_2} e^{-2\sqrt{k_2}(R-\rho)}
+o\left(e^{-2\sqrt{k_2}(R-\rho)}\right).
\end{align*}
Again, being $U_\rho$ a solution of \eqref{urho},
\begin{align*}
(B)&=-\int_\frac{R}{2}^R (-U''_\rho + \o U_\rho )V_\rho  \, dr
-\frac 1{4}\left( \int_\frac{R}{2}^R   U_\rho^2 \, dr\right)^2
\int_\frac{R}{2}^R U_\rho V_\rho \, dr + \int_\frac{R}{2}^R
U_\rho^{p}V_\rho \, dr
\\
&\quad-V_\rho(R)U'_\rho(R) +o\left(e^{-\sqrt{k_2}\frac R4}\right)
\\
&=\frac 14 \left(\int_{-\infty}^{+\infty}   U^2 \, dr\right)^2
\int_\frac{R}{2}^R U_\rho V_\rho  \, dr -\frac 1{4}\left(
\int_\frac{R}{2}^R   U_\rho^2 \, dr\right)^2  \int_\frac{R}{2}^R
U_\rho V_\rho \, dr
\\
&\quad+\sqrt{k_2} e^{-2\sqrt{k_2}(R-\rho)}
+o\left(e^{-2\sqrt{k_2}(R-\rho)}\right)
\\
&=\sqrt{k_2} e^{-2\sqrt{k_2}(R-\rho)}
+o\left(e^{-2\sqrt{k_2}(R-\rho)}\right).
\end{align*}
Moreover, since $V_\rho$ is a solution of \eqref{vrho},
\begin{align*}
(C)&=\frac{1}{2}\int_\frac{R}{2}^R \left(-V''_\rho +\o V_\rho
\right) V_\rho\, dr +\frac 1{8}\left( \int_\frac{R}{2}^R U_\rho^2
\, dr\right)^2  \int_\frac{R}{2}^R   V_\rho^2 \, dr +\frac
12V_\rho(R)V'_\rho(R)
\\
&\quad+o\left(e^{-\sqrt{k_2}\frac R4}\right)
\\
&=-\frac 18 \left(\int_{-\infty}^{+\infty}   U^2 \, dr\right)^2
\int_\frac{R}{2}^R V_\rho^2  \, dr +\frac 18\left(
\int_\frac{R}{2}^R   U_\rho^2 \, dr\right)^2  \int_\frac{R}{2}^R
V_\rho^2 \, dr
\\
&\quad+\frac 12\sqrt{k_2} e^{-2\sqrt{k_2}(R-\rho)}
+o\left(e^{-2\sqrt{k_2}(R-\rho)}\right)
\\
&=\frac 12\sqrt{k_2} e^{-2\sqrt{k_2}(R-\rho)}
+o\left(e^{-2\sqrt{k_2}(R-\rho)}\right).
\end{align*}
Finally, since
\[
\left| (U_\rho-V_\rho)^{p+1} - U_\rho^{p+1} +(p+1)U_\rho^{p}
V_\rho\right| \le \frac{p(p+1)}{2}U_\rho^{p-1} V_\rho^2,
\]
we have
\[
(D)=o\left(e^{-2\sqrt{k_2}(R-\rho)}\right).
\]
Then the conclusion follows.
\end{proof}

The proof of Theorem \ref{main} is completed with the following
proposition.

\begin{proposition}
Let us define $\Phi: \mathcal{I}_R \to \R$ as
$\Phi(\rho)=I_R(z_\rho+w_\rho)$. Then,
\[
\Phi(\rho)=2 \pi J(U)\rho +2 \pi \sqrt{k_2} \rho \
e^{-2\sqrt{k_2}(R-\rho)} +o\left(\rho \
e^{-2\sqrt{k_2}(R-\rho)}\right).
\]
Moreover, $\Phi$ attains a global minimum in the interior of
$\mathcal{I}_R$ for $R$ sufficiently large.
\end{proposition}

\begin{proof}

First, we use Taylor expansion to get, for some $\xi_\rho \in
(0,1)$,
\begin{align*}
I_R(z_\rho+w_\rho) &= I_R(z_\rho) + I_R'(z_\rho)[w_\rho] + \frac 1 2
I_R''(z_\rho + \xi_\rho w_\rho)[w_\rho, w_\rho] \\
& =I_R(z_\rho) + I_R'(z_\rho)[w_\rho] + \frac 1 2 I_R''(z_\rho)[w_\rho,
w_\rho] + \frac 1 2 (I_R''(z_\rho+\xi_\rho w_\rho)- I_R''(z_\rho))[w_\rho,
w_\rho].
\end{align*}
By using Lemmas \ref{pr:I'}, \ref{3.2}, \ref{3.3}, we obtain that:
$$ I_R(z_\rho+w_\rho) = I_R(z_\rho) + O(\|w_\rho\|^2).$$
Proposition \ref{auxiliary} and \eqref{alphabeta} imply that the
error $\|w_\rho\|^2$ is negligible, so that
$$ \Phi(\rho)=2 \pi J(U)\rho +2 \pi \sqrt{k_2}
\rho \ e^{-2\sqrt{k_2}(R-\rho)} +o\left(\rho \
e^{-2\sqrt{k_2}(R-\rho)}\right).$$

For the sake of brevity, we set
\begin{align*}
\rho_\a&=R-\frac{\a}{2\sqrt{k_2}}\log R,
\\
\rho_1&=R-\frac{1}{2\sqrt{k_2}}\log R,
\\
\rho_\b&=R-\frac{\b}{2\sqrt{k_2}}\log R.
\end{align*}

We will prove that $\Phi$ has a minimum in the interior of $\I_R$,
showing that $\Phi(\rho_\b)-\Phi(\rho_1) \to +\infty$ and
$\Phi(\rho_\a)-\Phi(\rho_1) \to +\infty$, as $R\to +\infty$.
Indeed
\begin{align*}
\Phi(\rho_\b)-\Phi(\rho_1) &=\frac{\pi
(1-\b)J(U)}{\sqrt{k_2}}\log R + O(1),
\\
\Phi(\rho_\a)-\Phi(\rho_1) &=2 \pi \sqrt{k_2}
R^{1-\a}+o(R^{1-\a}).
\end{align*}

\end{proof}

\begin{remark}
The solution that we have found corresponds to a minimum of the
reduced functional, and recall that $w_{k_2}$ is also a minimizer for the limit functional $J$. Therefore, it is reasonable to
think that our solution corresponds to a local minimum of $I_R$.

Moreover, its energy diverges negatively as $R \to +\infty$, so it could be the
global minimizer for $I_R$. Indeed, by the arguments in
\cite{noi}, the global minimizer of $I_R$ must contain a term like
$U_\rho(r)$, for some $\rho \sim R$. However, it is not clear if
the global minimizer could also have a ``localized part".

\end{remark}

\end{document}